\documentclass[article,11pt]{amsart}    
\usepackage{amsmath}
\usepackage{amsfonts} 
\usepackage{amssymb}
\usepackage{dsfont}
\usepackage{paralist}
\usepackage{graphics} 
\usepackage{epsfig} 
\usepackage{graphicx}
\usepackage{epstopdf}

\usepackage[colorlinks=true]{hyperref}
\hypersetup{urlcolor=blue, citecolor=red}
\newtheorem{Proposition}{Proposition}[section]

\newtheorem{Lemme}{Lemma}[section]
\newtheorem{Theoreme}{Theorem}[section]

\newcommand{\sign}{\text{sign}}

\def \mE{\mathcal{E}}

\def \R{\mathbb{R}}
\def \N{\mathbb{N}}

\def \ds{\displaystyle}

\renewcommand{\phi}{\varphi}
\newcommand\restr[2]{{
  \left.\kern-\nulldelimiterspace 
  #1 
  \vphantom{\big|} 
  \right|_{#2} 
}}

\begin{document}

\title[]{Finite time blow-up for a  nonlinear parabolic equation with
smooth coefficients}  
\author{Oscar Jarr\'in}
\address{Escuela de Ciencias Físicas y Matemáticas, Universidad de Las Américas, V\'ia a Nay\'on, C.P.170124, Quito, Ecuador}
\email{(corresponding author) oscar.jarrin@udla.edu.ec}
\author{Gast\'on Vergara-Hermosilla}
\address{Institute for Theoretical Sciences, Westlake University, Hangzhou,  People's Republic of China}
\email{ gaston.v-h@outlook.com}
\date{\today}


\maketitle

\begin{abstract} In this article, we consider an $n$-dimensional parabolic partial differential equation with a smooth coefficient term in the nonlinear gradient term. This equation was first introduced and analyzed in [E. Issoglio, On a non-linear transport-diffusion equation with distributional coefficients, Journal of Differential Equations, Volume 267, Issue 10 (2019)], where one of the main \emph{open questions} is the possible \emph{finite-time blow-up} of solutions. Here, leveraging a \emph{virial-type estimate}, we provide a \emph{positive answer} to this question within the framework of smooth solutions.  
\end{abstract}

\section{Introduction and main result}
This article investigates the finite-time blow-up of smooth solutions to the following nonlinear parabolic equation:
\begin{equation}\label{EquationIntro0}
\begin{cases}
\partial_tu=\Delta u + | \nabla u|^2 \mathfrak{b},\\[3pt]
u(0,x)=u_0(x),
\end{cases}
\end{equation}
where, given a time $T>0$ and $n\geq 1$, $u:[0,T]\times \R^n \to \R$ denotes the unknown solution, $\mathfrak{b}:[0,T]\times \R^n \to \R$ is a prescribed coefficient term, and $u_0: \R^n \to \R$ stands for the initial datum. 

\medskip

The primary motivation for studying this class of models stems from different physical and probabilistic contexts. In the latter, specific choices of singular coefficients $\mathfrak{b}(t,x)$ are particularly relevant when modeling the realization of random noise.   We refer to the interested reader to \cite{Hinz,Hinz2,Issoglio13} for further details.

\medskip

In the deterministic case, the equation (\ref{EquationIntro0}) was first introduced and studied in \cite{Issoglio19}. 
In order to describe some of the main results in such paper, it is convenient to recall that, for any parameter $\gamma \in \R$, the H\"older-Zygmund  spaces $\mathcal{C}^\gamma(\R^n)$ can be characterized through the general framework of Besov spaces yielding the equivalence $\mathcal{C}^\gamma(\R^n) \approx  B^{\gamma}_{\infty, \infty}(\R^n)$ (see for instance \cite[Chapter 2]{Triebel}). 
Within this framework, the main contribution of \cite{Issoglio19} is to establish the existence and uniqueness  of  \emph{local-in-time} solutions $u\in \mathcal{C}_t\, \mathcal{C}^{\alpha+1}_x$, which arise   from initial data in H\"older-Zygmund  spaces of positive order $\mathcal{C}^{\alpha+1}(\R^n)$, for appropriate values of $\alpha>0$, and coefficient terms in H\"older-Zygmund  spaces of negative order $\mathcal{C}^{\beta}(\R^n)$, for suitably chosen $\beta<0$.  In particular, as mentioned above, this  singular coefficient terms are motivated by a  probabilistic interpretation. 

\medskip

More precisely, such solutions are  constructed by using the following (formally equivalent) mild formulation: 
\begin{equation*}
    u(t,\cdot)=h_t\ast u_0 + \int_{0}^{t}h_{t-\tau}\ast \big( |\nabla u(\tau,\cdot)|^2\, \mathfrak{b}(\tau,\cdot)\big) \, d\tau,
\end{equation*}
where $h_t$ denotes the well-known heat kernel. In this formulation, we highlight  that the  nonlinear expression $|\nabla u(t,x)|^2$, together with the coefficient term $\mathfrak{b}(t,x)$, make the nonlinear term delicate to handle. To overcome this 
issue,  
some  suitably constraints on the parameters $\beta <0<\alpha$ are required. Specifically, the technical relationship  $-1/2<\beta<0<\alpha<1$, together with the assumptions that $u_0\in \mathcal{C}^{\alpha+1}_x$  and $\mathfrak{b}\in \mathcal{C}_t \mathcal{C}^{\beta}_x$, allowed the author of \cite{Issoglio19} to prove a  suitably   control on the nonlinear term above.  Consequently, using a  contraction principle, the solution $u \in \mathcal{C}_t \, \mathcal{C}^{\alpha+1}_x$ to the problem above is rigorously obtained  for a time $T_0>0$, which is  sufficiently small respect with the quantity $\|u_0\|_{\mathcal{C}^{\alpha+1}}$ and the parameters $\alpha, \beta$. For more details, we refer to \cite[Theorem $3.7$]{Issoglio19}.

\medskip

Due to the restrictions on the existence time $T_0$, which are essentially  imposed by the nonlinear term $|\nabla u|^2$, one  of the main \emph{open questions} in \cite{Issoglio19} concerns the global-in-time or blow-up in finite time of solutions to the equation (\ref{EquationIntro0}). In this context, in \cite[Proposition $3.13$]{Issoglio19}, it is proven a first blow-up criterion, essentially stating that if blow-up occurs then it must holds in a finite time $t_*>0$. 

\medskip

Thus, with the aim of studying the possible blow-up of solutions,  in \cite{Chamorro} the authors analyzed the following toy model for equation (\ref{EquationIntro0}):
\begin{equation}\label{EquationIntro}
\begin{cases}
\partial_t u=-(-\Delta)^{\frac{\alpha}{2}} u +|(-\Delta)^{\frac{1}{2}} u |^2\ast \mathfrak{b}, \qquad  0<\alpha \leq 2,\\[3pt]
\mathfrak{u}(0,x)=\mathfrak{u}_0(x). 
\end{cases}
\end{equation}
Here, the classical Laplacian operator is substituted by its fractional power, which is defined at the Fourier level by the symbol $|\xi|^\alpha$. 
However, the main featured of this model is given by the \emph{modified nonlinear term}, where the  expression $|\nabla u|^2$ becomes the nonlocal expression with the same derivative order $|(-\Delta)^{\frac{1}{2}} u|^2$; and the product with the coefficient term $\mathfrak{b}(t,x)$ becomes a product of convolution.

\medskip

This somehow strong modification on the nonlinear term, allow the authors of \cite{Chamorro} to apply a  Fourier-based method to prove the blow-up in finite time of  solutions. First, for any initial datum $u_0 \in H^1(\R^n)$,  in \cite[Theorem $1$]{Chamorro} the authors construct local-in-time  mild solutions $u\in L^\infty([0,T_0], H^1(\R^n))$ to the equation (\ref{EquationIntro}). These solutions are obtained from an interesting relationship between the fractional power $\alpha$ and the regularity in spatial variable assumed on the coefficient $\mathfrak{b}(t,x)$. 
More precisely, the  weak smoothing effects of the fractional Laplacian operator  in the case  $0<\alpha<1$ are compensated by stronger regular coefficients $\mathfrak{b} \in L^\infty_t H^\gamma_x$ with $0\leq 1-\alpha \leq \gamma<1$. In contrast, stronger smoothing effects in the case when $1<\alpha \leq 2$ allow to consider singular coefficients  $\mathfrak{b} \in L^\infty_t H^{-\gamma}_x$, with $0\leq \gamma < \alpha -1$.
%
%
Having constructed mild solutions $u\in L^\infty_tH^1_x$,   the main result of  \cite{Chamorro}  states that, in the Fourier variable, one can define  well-prepared initial data $\widehat{u_0}(\xi)$ and coefficient terms $\widehat{\mathfrak{b}}(t,\xi)$ to obtain the blow-up of the associated solution $u(t,\cdot)$ at a finite time $t_{*}>0$, which  explicitly depends on the fractional power $\alpha$ by the expression   $t_{*}=\frac{\ln(2)}{2^\alpha}$. 

\medskip

This approach is mainly inspired by the method of  Montgomery-Smith  \cite{Montgomery-Smith}, which was firstly considered to study the blow-up of the well-known cheap Navier-Stokes equations. See also \cite[Section $11.2$]{Lemarie-Rieusset} and \cite{Cheskidov,Gallagher-Paicu} for more details and related results. In this method, since the operator $(-\Delta)^{\frac{1}{2}}$ has the positive symbol $|\xi|$ in the Fourier level, along with the well-known fact that the convolution product with $\mathfrak{b}(t,x)$ becomes a pointwise product with $\widehat{\mathfrak{b}}(t,\xi)$, the authors of \cite{Chamorro} are able to derive suitably lower bounds for the mild formulation of the solution    in the Fourier variable:
\begin{equation*}
  \widehat{u}(t,\xi)=e^{-|\xi|^\alpha\, t}\,\widehat{u_0}(\xi)+\int_{0}^{t}e^{-|\xi|^\alpha (t-\tau)} \Big(\big(|\xi|\widehat{u}\ast |\xi|\widehat{u}\big)\, \widehat{\mathfrak{b}}\Big)(\tau,\xi)\,d\tau. 
\end{equation*}
These lower bounds, involving well-chosen weight functions defined in the separate variables $\xi\in \R^n$ and $t>0$, yield the finite blow-up of the expression $\|u (t,\cdot)\|_{H^1}$, when $t\to t_{*}$.  

\medskip

However, the method of Montgomery-Smith seems to break down for the original equation (\ref{EquationIntro0}). Essentially, in contrast to the fractional Laplacian operator $(-\Delta)^{\frac{1}{2}}$, now  each component   of the gradient operator $\nabla$ in the nonlinear term has the imaginary symbol ${{\bf i}\xi_j}$. Consequently, when deriving lower bound on the expression $\widehat{u}(t,\xi)$, one loses the sign information  in the nonlinear term, which is one  of the key elements in this method.  Consequently, the previous results \cite{Chamorro,Issoglio19} left open the possible blow-up of solution to the equation  (\ref{EquationIntro0}). 

\medskip

In this context, this article gives a positive answer to this problematic of equation (\ref{EquationIntro0}). Our main contribution is to show that, in this equation,  one can define suitable chosen initial data $u_0(x)$ and coefficient terms $\mathfrak{b}(t,x)$ which, together with the nonlinear effects of the term $|\nabla u|^2$, produces a blow-up in finite time in the solution. 

\medskip

To achieve such result, we apply  a different approach   mainly inspired by \cite{Cordoba}, in the setting of a nonlocal transport equation, and \cite{Dong-Du-Li}, in the setting of the fractional  Burgers equation. As we will elaborate in more detail in the following section, this method fundamentally relies on  the differential formulation of equation (\ref{EquationIntro0}), where, for a  well-chosen weight function, we are able to derive a \emph{virial-type estimate} on the solution yielding its blow-up in finite time.

\subsection*{The main result} To implement the method mentioned above, we will work with smooth solutions of equation (\ref{EquationIntro0}). Specifically, we will require certain local boundedness and continuity properties of their derivatives with respect to the spatial variable. For the completeness of this article, we will  start by constructing these solutions, which arise from initial data  and coefficient terms in Sobolev spaces. 
\begin{Proposition}\label{Main-Proposition}
For $s>n/2 +3$, let $u_0 \in H^s(\R^n)$ be an initial datum. Additionally, assume that the coefficient $\mathfrak{b}(t,x)$ satisfies
\begin{equation}\label{Assumption-coeffiecient-Sobolev}
   \mathfrak{b}\in \mathcal{C}\cap L^\infty\big( [0,+\infty[, H^{s-1}(\R^n)\big). 
\end{equation}
Then, there exists a time $T_0>0$, which depends on $u_0$ and $\mathfrak{b}$, and there exists a function 
\[ u \in \mathcal{C}\big([0,T_0], H^s(\R^n)\big), \]
which is the unique solution to the equation (\ref{EquationIntro0}).
\end{Proposition}

By the assumption $s>n/2+3$,  the  Sobolev embedding $H^s(\R^n)\subset \mathcal{C}^{3}(\R^n)$ holds. Then,  it follows that   $u_0 \in \mathcal{C}^3(\R^n)$, $\mathfrak{b}\in \mathcal{C}^{1}([0,+\infty[, \mathcal{C}^2(\R^n)$ and $u\in \mathcal{C}([0,T_0],\mathcal{C}^3(\R^n))$. Particularly, the solution $u(t,x)$ verifies the the equation (\ref{EquationIntro0}) in the classical sense. It is worth mentioning that these solutions could also be obtained using the more general framework of Besov spaces. Additionally, we do not study the optimality of the relationship $s>n/2+3$ for obtaining smooth solutions. As explained, our main objective is to focus directly on the blow-up issue.

\medskip

Now, with this information at hand, we briefly describe  the  strategy to prove the finite blow-up of smooth solutions $u(t,x)$ to  equation (\ref{EquationIntro0}). In this equation,   the quadratic term $|\nabla u|^2$ suggests to introduce the vector field $\mathrm{v}(t,x):= \nabla u(t,x)$, which  is  a solution of the following nonlinear  system:
\begin{equation}\label{Main-System-n}
\begin{cases} 
\partial_t \mathrm{v}= \Delta \mathrm{v}+\nabla \big(|\mathrm{v}|^2\, \mathfrak{b} \big), \\
\mathrm{v}(0,\cdot)= \nabla u_0. 
\end{cases}
\end{equation}
Since $u\in \mathcal{C}_t \, H^s_x$ with $s>n/2+3$, it follows that $\mathrm{v}\in \mathcal{C}_t \, H^{s-1}_x \subset \mathcal{C}_t\, \mathcal{C}^2_x$. Consequently, the vector field $\mathrm{v}(t,x)$ also solves this system in the classical sense. 

\medskip

Using the information $\mathrm{v}\in  \mathcal{C}_t\, \mathcal{C}^2_x$, the key idea in our method is to first show that the solution $\mathrm{v}(t,x)$ blows-up in a finite time $t_{*}>0$, which in turn yields the blow-up of the quantity $\| u(t,\cdot)\|_{H^s}$ as $t\to t_{*}$. 

\medskip

The blow-up phenomenon of the solution $\mathrm{v}(t,x)$  arises from the nonlinear effects of the term $\nabla \big(|\mathrm{v}|^2\, \mathfrak{b} \big)$, along with  certain  suitable  conditions on the coefficient $\mathfrak{b}(t,x)$. Specifically, for simplicity, we will assume that this coefficient depends only on the spatial variable. Additionally,  for  any $x=(x_1,\cdots, x_i, \cdots x_n)\in \R^n$, we assume that 
\begin{equation}\label{Condition-b}
\mathfrak{b}(x)=\prod_{i=1}^{n} \mathfrak{b}_{i}(x_i),\\
\end{equation}
where each term term $\mathfrak{b}_{i}(x_i)$ is a  \emph{non-constant}  function satisfying: 
\begin{equation}\label{Condition-bi}
 \mathfrak{b}_i \in  \mathcal{S}(\R), \qquad \mathfrak{b}_i(x_i)\geq 0 \quad \text{for any $x_i \in \R$}, \qquad  \mathfrak{b}_i(0)=0.
\end{equation}
Note that from the expression (\ref{Condition-b}) it holds that  $\mathfrak{b}\in \mathcal{C}\cap L^\infty([0,+\infty[, \mathcal{S}(\R^n))$. Consequently, this coefficient verifies the assumption (\ref{Assumption-coeffiecient-Sobolev}).

\medskip

In expressions (\ref{Vector-W})–(\ref{Component-w}) below, we define a suitable vector field $\mathrm{w}(x)$ localized over the compact  domain $[-1,1]^n$, which is  essentially chosen to  illustrate our main computations. Then, given the solution $\mathrm{v}(t,x)$ of the system (\ref{Main-System-n}), the virial-type estimates that we shall derive  allow us to study the time evolution of the functional
\[ I(t):=\int_{[-1,1]^n} \mathrm{v}(t,x)\cdot \big(\mathfrak{b}(x)\, \mathrm{w}(x)\big)dx. \]

\medskip

The assumption (\ref{Condition-b}) on the coefficient $\mathfrak{b}(x)$ yields that this functional satisfies an ordinary differential inequality with a quadratic nonlinearity. From this, the expression $I(t)$  blows up in finite time as long as the initial data $u_0(x)$ satisfies the conditions:
\begin{equation}\label{Condition-u0-1}
 \int_{[-1,1]^n} \nabla u_0(x) \cdot \big(\mathfrak{b}(x)\mathrm{w}(x)\big)  \, dx >0,  \end{equation}
 and
\begin{equation}\label{Condition-u0-2}
\left( \int_{[-1,1]^n} \nabla u_0(x) \cdot \big(\mathfrak{b}(x)\mathrm{w}(x)\big)  \, dx \right)^2 \geq \mathfrak{C}\times \begin{cases}\vspace{2mm}\|u_0\|_{H^s}, \quad \text{when}\, \,\, \|u_0\|_{H^s}\leq 1, \\
\|u_0\|^2_{H^s}, \quad \text{when}\,\,\, \|u_0\|_{H^s}>1,
\end{cases}
\end{equation}
for some constant $\mathfrak{C}>0$ essentially  depending on $n$, the (technical) parameter $\kappa$  introduced in expression (\ref{Component-w}) below, and the quantities $\| \mathfrak{b}\|_{L^\infty}$ and $\| \nabla \mathfrak{b}\|_{L^\infty}$.  In this context, our main result reads as follows: 

\begin{Theoreme}\label{Main-Th} Under the same assumptions as in Proposition \ref{Main-Proposition}, suppose that the coefficient $\mathfrak{b}(t,x)$ satisfies conditions (\ref{Condition-b})-(\ref{Condition-bi}), and the initial datum $u_0(x)$ adheres to conditions (\ref{Condition-u0-1})-(\ref{Condition-u0-2}).

\medskip

Then, for any given number $m_{*}\gg 1$ there exists a  finite time $t_{*}=t_{*}(m_{*})>0$, depending on $m_{*}$ and explicitly provided  in the expression (\ref{Time-blow-up}) below, such that  the quantity $\| u(t,\cdot)\|_{H^s}$ blows-up at $t_{*}$.
\end{Theoreme}

The following comments are in order. First, as  previously mentioned, we show that smooth solutions of equation (\ref{EquationIntro0}) can develop a blow-up phenomenon in the case of a family of well-chosen coefficient terms satisfying assumption (\ref{Condition-b}) and a family of well-chosen initial data satisfying (\ref{Condition-u0-1}) and (\ref{Condition-u0-2}). In particular, this type of assumption on initial data appears in previous blow-up results \cite{Arnaiz, Dong-Du-Li, Jarrin-Vergara-Hermosilla} for different one-dimensional  models. In this setting, our result can also be seen as a generalization of this method to the $n$-dimensional case. 
On the other hand, it is interesting to point out that a functional similar to $I(t)$ was previously considered in  the study of the Dyadic model of the Navier-Stokes equations in \cite{Cheskidov}.

\medskip

In the particular case of equation (\ref{EquationIntro}) when $\mathfrak{b}\equiv 1$, it is interesting to make a comparison with the model:  
\[ \partial_t u = \Delta u + |\nabla u |^2. \]
In  \cite{Gilding}, the authors prove the existence of a unique \emph{global-in-time} explicit solution arising from any bounded and continuous  initial datum. This fact shows that  \emph{non-constant}  coefficient terms $\mathfrak{b}$ plays a substantial role in the time dynamics of equation (\ref{EquationIntro}). See also  \cite{Fila,Souplet2002} for related results.

\medskip


\medskip

{\bf Organization of the rest of the article}. In Section \ref{Sec-Prelimilaries}, we recall some well-known tools that we will use in our proofs. Section \ref{Sec-LWP} is devoted to providing a proof of Proposition \ref{Main-Proposition}. Finally, in Section \ref{Sec-Blow-Up}, we prove our main result stated in Theorem \ref{Main-Th}.

\section{Preliminaries}\label{Sec-Prelimilaries}
In order to keep this paper reasonably self-contained, in this section we collect some results for the proof of our main results.
Thus, we begin by presenting well-known facts about Sobolev spaces $H^s(\R^n)$.

\begin{Proposition}[$H^s$ estimate]\label{prop:3}
	Let $s_1, s_2 \geq 0$. Then, there is a constant $C > 0$, which depends on the dimension $n \in \N^*$ and the parameter $s_2$, such that 
$$\|h_t*\varphi\|
_{H^{s_1 + s_2}} \leq C \left(1 + t^{-
s_2}\right) \|\varphi\|_{H^{s_1}}.$$
\end{Proposition}
A proof of this result can be consulted in \cite{Jarrin-L}.

\begin{Proposition}[Banach algebra]\label{prop:4}
	Let $f,g\in H^s(\R^n)$, for $s>n/2$. Then, $H^s(\R^n)$ is a Banach algebra. In particular, there exists $C>0$, which depends on the dimension $n$ and the parameter $s$,  such that
    \begin{equation*}
        \| f g \|_{H^{s}}
        \leq C
        \| f \|_{H^{s}}
        \| g \|_{H^{s}}.
    \end{equation*}
    
\end{Proposition}
A proof of this result can be consulted in the appendix of book \cite{Bedrossian-Vicol}.

\medskip

Finally, we recall a classical abstract result which will be considered to construct an unique (local in time) solution of \eqref{EquationIntro0}. 
\begin{Theoreme}[Banach-Picard  principle]\label{BP_principle}
Consider a Banach space $(\mathcal{E},\|\cdot \|_\mathcal{E} )$ and a bounded bilinear application ${B}: \mathcal{E} \times \mathcal{E} \longrightarrow \mathcal{E}$: 
$$\| B(e,e)\|_{\mathcal{E}}\leq C_{ B}\|e\|_\mathcal{E}^2.$$
Given $e_0\in \mathcal{E}$  such that $\|e_0\|_\mathcal{E} \leq \gamma$ with $0<\gamma< \frac{1}{4C_{B}}$, then the equation 
$$e = e_0 -   B(e,e),$$
admits a unique solution $e\in \mathcal{E}$ which satisfies $\| e \|_\mathcal{E} \leq 2 \gamma$.
\end{Theoreme}

\section{Local wellposedness result: proof of Proposition \ref{Main-Proposition}}\label{Sec-LWP}

In order to prove Proposition \ref{Main-Proposition}, in this section
we will use the Banach-Picard principle for a time $0 < T < +\infty$ small enough. 
To this end we consider the Banach space
$$
\mE_{T}
=
 \mathcal{C}([0,T], H^s(\R^n)),
$$
which is endowed with its natural norm. Thus, with this framework at hand,
we begin by considering the following classical estimate for the initial data
\begin{equation}\label{est. initial data}
\| h_t \ast u_0 \|_{ \mE_{T}}
\leq C_1
\| u_0 \|_{H^s},
\end{equation}
where $C_1>0$ is a constant.  Then, we study the nonlinear term: 
%
\begin{Lemme}\label{lemma 2}
Let $ \mathfrak{b} \in \mathcal{C}\cap L^\infty([0,+\infty], H^{s-1}(\R^n))$, with $s>{n}
/{2}+3$.
Then, there exists a constant $C_\mathcal{B}(T)>0$, which depends on the time $T$ and the norm $\| \mathfrak{b}\|_{L^\infty_t H^{s-1}_x}$, such that
\begin{equation}\label{est. bilinear term}
\left\| 
\int_{0}^{t} h_{t-s}\ast |\nabla u|^2  \mathfrak{b} (s,\cdot) ds
\right\|_{ \mE_{T} }
\leq C_\mathcal{B}(T)\,
\| u \|_{\mE_{T}}^2.
\end{equation}
\end{Lemme}
\begin{proof}
To start, we consider the $H^s$-norm to the term $\ds{\int_0^t h_{t-s} * | \nabla u |^2  
\mathfrak{b} (s,\cdot)
ds} $
to get
\[
\left\|
\int_0^t h_{t-s} * | \nabla u |^2  
\mathfrak{b} (s,\cdot)
ds 
\right\|_{ H^s }
\leq \int_0^t \| h_{t-s} * | \nabla u |^2 \mathfrak{b} (s,\cdot) \|_{H^s} ds.
\]
Then, by using  Proposition \ref{prop:3}
with $s_1 = s-1$ and $s_2 = 1$, 
we can write 
\[
\int_0^t 
\|  h_{t-s} * | \nabla u |^2 \mathfrak{b} (s,\cdot)\|_{H^s} ds
\leq
\int_0^t  
\left( C+  \frac{C}{(t-s)^{\frac{1}{2}}} \right) 
\| | \nabla u |^2 \mathfrak{b} (s,\cdot)\|_{ H^{s-1}}ds.
\]
Now, since $H^\sigma(\R^n)$ is a Banach algebra for $\sigma>{n}/{2}$ (see Proposition \ref{prop:4} with $\sigma=s-1$), 
we obtain
\[
\left\|
\int_0^t h_{t-s} * | \nabla u |^2  
\mathfrak{b} (s,\cdot)
ds 
\right\|_{ H^s }
\leq
\int_0^t  \left( C+  \frac{C}{(t-s)^{\frac{1}{2}}} \right) 
\| \nabla u   (s,\cdot)\|_{ H^{s-1}}^2
\|   \mathfrak{b} (s,\cdot)\|_{ H^{s-1}}
ds.
\]
Thus, by taking the $L^{\infty}$-norm on the time interval $[0,T]$, we get 
\[
\left\|
\int_0^t h_{t-s} * | \nabla u |^2  
\mathfrak{b} (s,\cdot)
ds 
\right\|_{ \mE_{T} }
\leq
\left(
\int_0^T C+ \frac{C}{(T-s)^{\frac{1}{2}}} 
ds
\right)
\|   \mathfrak{b} \|_{L^\infty_t H^{s-1}_x}
\|   u    \|_{ \mE_T}^2,
\]
and then, by considering our assumption on  $ \mathfrak{b}$, we conclude the estimate
\[
\left\| 
\int_{0}^{t} h_{t-s}\ast |\nabla u|^2  \mathfrak{b} (s,\cdot) ds
\right\|_{ \mE_{T} }
\leq C
\big(T + \sqrt{T}\big) \| \mathfrak{b}\|_{L^\infty_t H^{s-1}_x}\,
\|   u  \|_{ \mE_T}^2.\]
Therefore, by considering $C_\mathcal{B}(T)= C \big(T + \sqrt{T} \big)\| \mathfrak{b}\|_{L^\infty_t H^{s-1}_x}$, we deduce \eqref{est. bilinear term} and  we conclude the proof of this  lemma.
\end{proof}
\subsection*{End of the proof}
Gathering together the hypothesis assumed in the statement of  Proposition \ref{Main-Proposition} and the estimates obtained in estimates (\ref{est. initial data}) and (\ref{est. bilinear term}),  we conclude the existence of a time  $0<T_0<+\infty$ such that  $4C_1 C_\mathcal{B}(T_0)\| u_0\|_{H^s}<1$. Thus, by applying the Banach-Picard  principle stated in Theorem \ref{BP_principle} we obtain a unique (local in time) solution of \eqref{EquationIntro0}. 
This concludes the proof of
Proposition \ref{Main-Proposition}.

\section{Blow-up result: proof of Theorem \ref{Main-Th}}\label{Sec-Blow-Up}
For $s>n/2+3$ and $T_0>0$, let $u\in \mathcal{C}([0,T_0], H^s(\R^n))$ be the smooth   solution to the equation (\ref{EquationIntro0}), which is obtained in Proposition \ref{Main-Proposition}, from any initial datum $u_0 \in H^s(\R^n)$. In order to prove Theorem \ref{Main-Th}, we argue by contradiction and we assume the following key hypothesis
\begin{equation*}
(H): \,\, \text{the solution $u(t,x)$ can be extended to the entire interval of time $[0,+\infty[$}.     
\end{equation*}

In this framework, we define the vector field $\mathrm{v}(t,x):=\nabla u(t,x)$, which verifies the  system (\ref{Main-System-n}), and  from hypothesis $(H)$, it exists globally-in-time.

\medskip

Now, for any  $x=(x_1,\cdots, x_i, \cdots, x_n)\in \R^n$, we define the following vector field of weights
\begin{equation}\label{Vector-W}
    \mathrm{w}(x)= \big(w_1(x_1),\cdots, w_i(x_i),\cdots, w_n(x_n)\big),
\end{equation}
with 
\begin{equation}\label{Component-w}
     w_i(x_i):= \begin{cases}\vspace{2mm}
    \sign(x_i)(|x_i|^{-\kappa}-1), \quad -1<x_i<1, \, \, \,  x_i\neq 0, \quad 0<\kappa<1, \\
    0, \quad \text{otherwise}.
    \end{cases}
\end{equation}
From \eqref{Component-w} we deduce that each component is an odd function supported in the interval $[-1,1]$, and thus we can see that $\text{supp}(\mathrm{w})\subseteq[-1,1]^n$.

\medskip

Now, considering the coefficient $\mathfrak{b}(x)$ defined in  (\ref{Condition-b}) and (\ref{Condition-bi}), we  take the $\R^n$-inner product in each term of the system (\ref{Main-System-n}) with  $\mathfrak{b}(x)\,\mathrm{w}(x)$.
Then,  integrating  over $\R^n$ and noting that $\text{supp}(\mathrm{w})\subseteq [-1,1]^n$, we obtain 
\begin{equation}\label{Main-Identity-n}
\begin{split}
\frac{d}{dt}\left(\int_{[-1,1]^n} \mathrm{v}(t,x)\cdot \big( \mathfrak{b}(x)\mathrm{w}(x) \big)\, dx\right)= &\,  \int_{[-1,1]^n} \Delta \mathrm{v}(t,x)\cdot \big( \mathfrak{b}(x)\mathrm{w}(x) \big)\, dx \\
&\, + \int_{[-1,1]^n} \nabla (|\mathrm{v}(t,x)|^2\, \mathfrak{b}(x))\cdot \big( \mathfrak{b}(x)\mathrm{w}(x) \big)\, dx.
\end{split}
\end{equation}

Thus, in the following technical propositions, we derive a lower bound for each term on the right-hand side of \eqref{Main-Identity-n}. 

\medskip

To state the first proposition below, recall that the system (\ref{Main-System-n}) is directly related to the equation (\ref{EquationIntro0}) through the relationship $\mathrm{v}(t,x):=\nabla u(t,x)$. 
Thus, we will shift from one to the other as needed, without distinction.

\begin{Proposition}\label{Prop-Estim-Below-1-n}  For $s>n/2+3$, let $u_0 \in H^s(\R^n)$ be the initial datum in equation (\ref{EquationIntro0}).  Assume the hypothesis $(H)$, which yields that   for any given  time $T>0$,  the  solution $\mathrm{v}(t,x)$  of  the system  equation (\ref{Main-System-n}) is defined over the interval $[0,T]$. Additionally, let $0<\kappa<1$ be the parameter introduced in expression (\ref{Component-w}).  

\medskip

Then,  there exists a constant ${\bf C_1}(T)={\bf C_1}(n,\kappa,\|\mathfrak{b}\|_{L^\infty},T)>0$, 
such that for any time $0<t\leq T$ the following lower bound holds:
\begin{equation}\label{Main-Estim-1-n}
\int_{[-1,1]^n} \Delta \mathrm{v}(t,x)\cdot \big( \mathfrak{b}(x) \mathrm{w}(x)\big)\, dx  \geq - {\bf C}_1(T)\, \| u_0 \|_{H^s}.  
\end{equation}
\end{Proposition}
\noindent
\begin{proof} Considering (\ref{Vector-W}) and (\ref{Component-w}), since $0<\kappa<1$, we can directly deduce that $\mathrm{w}\in L^1(\R^n)$. Additionally, we can find a constant $C_{n,\kappa}>0$ such that $\| \mathrm{w}\|_{L^1}\leq C_{n,\kappa}<+\infty$. 

\medskip

Then, applying  H\"older inequalities and recalling that $\mathrm{v}(t,x):=\nabla u(t,x)$, we write
\begin{equation} 
    \begin{split}
         \left| \int_{[-1,1]^n} \Delta \mathrm{v}(t,x)\cdot \big(\mathfrak{b}(x)\mathrm{w}(x)\big)\, dx \right|
        \leq &\, \| \Delta \mathrm{v}(t,\cdot)\|_{L^\infty}\, \|\mathfrak{b}\|_{L^\infty}\, \| \mathrm{w}\|_{L^1}\\
        \leq &\,  \, C_{n,\kappa}\, \|\mathfrak{b}\|_{L^\infty}\,\| \Delta \mathrm{v}(t,\cdot)\|_{L^\infty}
    \end{split}
\end{equation}
Thus, using Sobolev embeddings (with a parameter $n/2<\sigma<s-3$), we obtain
\begin{equation}\label{Estim-1-n}
    \begin{split}
         \left| \int_{[-1,1]^n} \Delta \mathrm{v}(t,x)\cdot \big(\mathfrak{b}(x)\mathrm{w}(x)\big)\, dx \right|
        \leq &\,
        \, C_{n,\kappa}\, \|\mathfrak{b}\|_{L^\infty}\,C\, \| \Delta \mathrm{v}(t,\cdot) \|_{H^\sigma} \\
        \leq &  \, C_{n,\kappa}\, \|\mathfrak{b}\|_{L^\infty}\, C\,\| u (t,\cdot) \|_{H^{\sigma+3}}\\
        \leq  &\, C_{n,\kappa}\,\|\mathfrak{b}\|_{L^\infty}\, C\, \| u (t,\cdot) \|_{H^{s}}.
    \end{split}
\end{equation}

Here, we still need to estimate the term $\| u (t,\cdot) \|_{H^{s}}$.  Note that, from Hypothesis $(H)$,  the solution $u(t,x)$ to equation (\ref{EquationIntro0}) is extended to the interval $[0,+\infty[$ by the following standard iterative argument. For $k\in \mathbb{N}$, and suitably chosen times $0<T_k<T_{k+1}$, particularly $T_{k+1}-T_{k}$ must be sufficiently small, we consider the initial datum $u(T_k, \cdot)$, and applying again Picard's  fixed point scheme in the space $\mathcal{C}_t H^s_x$ we extend the solution to the interval of time $[T_{k}, T_{k+1}]$. Furthermore, we know that there exists a constant $c_k>0$ such that $\ds{\sup_{T_k \leq t \leq T_{k+1}}\|u(t,\cdot)\|_{H^s} \leq c_k \| u(T_k, \cdot)\|_{H^s}}$. 

\medskip

Iterating these estimates, we can find a constant $C_k>0$  big enough, where $\ds{C_k > \prod_{j=0}^{k}c_j}$, such that the following estimate holds
\[ \sup_{T_k \leq t \leq T_{k+1}}\| u(t,\cdot)\|_{H^s} \leq c_k \| u(T_k, \cdot)\|_{H^s} \leq C_k \, \| u_0 \|_{H^s}. \]
Thus, for any given time $T>0$, there  exists $k=k(T)\in \mathbb{N}$ such that 
\begin{equation}\label{Control-Solution}
\sup_{0\leq t \leq T}\| u(t,\cdot)\|_{H^s} \leq \sum_{j=0}^{k(T)} \, \sup_{T_{j}\leq t \leq T_{j+1}} \| u(t,\cdot)\|_{H^s} \leq \left( \sum_{j=0}^{k(T)}C_j\right)\|u_0\|_{H^s}=: C(T)\|u_0\|_{H^s}.
\end{equation}
With this control at hand, we return to (\ref{Estim-1-n}) to obtain
\begin{equation*}
\left| \int_{[-1,1]^n} \Delta \mathrm{v}(t,x)\cdot \big(\mathfrak{b}(x)\mathrm{w}(x)\big)\, dx \right|\leq C_{n,\kappa}\,\|\mathfrak{b}\|_{L^\infty} \, C(T)\| u_0\|_{H^s}=: {\bf C}_1(T)\| u_0\|_{H^s},
\end{equation*}
which yields the desired estimate (\ref{Main-Estim-1-n}).\end{proof}

\begin{Proposition}\label{prop 2 section main thm} Within the framework of Proposition \ref{Prop-Estim-Below-1-n},   for any $0\leq t\leq T$, there exists a constant ${\bf C}_2(T)={\bf C}_{2}\big(n,\kappa,\|\mathfrak{b}\|_{L^\infty},\| \nabla \mathfrak{b}\|_{L^\infty},T\big)>0$
 such that the following lower estimate holds
\begin{equation}\label{Main-Estim-2-n}
\begin{split}
 \int_{[-1,1]^n} \nabla \big(|\mathrm{v}(t,x)|^2\, \mathfrak{b}(x)\big)\cdot \big( \mathfrak{b}(x)\mathrm{w}(x)\big)\, dx \geq &\,  \frac{\kappa}{2^{n+1}}\left( \int_{[-1,1]^n} \mathrm{v}(t,x)\cdot \big(\mathfrak{b}(x)\mathrm{w}(x)\big)\,dx \right)^2 \\
 &\, - {\bf C}_2(T)\| u_0\|^2_{H^s}.
 \end{split}
\end{equation}
\end{Proposition}
\noindent
\begin{proof}
To begin, we consider  expression (\ref{Vector-W}), to write
\begin{equation}\label{Identity-Tech-0}
    \begin{split}
 &\,  \int_{[-1,1]^n} \nabla \big(|\mathrm{v}(t,x)|^2\, \mathfrak{b}(x)\big)\cdot \big(\mathfrak{b}(x) \mathrm{w}(x)\big)\, dx \\
 =&\, \sum_{i=1}^{n} \int_{[-1,1]^n} \partial_i \big(|\mathrm{v}(t,x)|^2\, \mathfrak{b}(x)\big) \big(\mathfrak{b}(x) w_i(x_i)\big)\, dx\\
 =&\, \sum_{i=1}^{n}  \int_{-1}^{1}\cdots \int_{-1}^{1} \partial_i \big(|\mathrm{v}(t,x)|^2\, \mathfrak{b}(x)\big) \big(\mathfrak{b}(x) w_i(x_i)\big)\,  dx_{1}\cdots dx_{i}\cdots dx_{n} \\
 =&\, \sum_{i=1}^{n} \int_{-1}^{1}\cdots \left( \int_{-1}^{1}  \partial_i \big(|\mathrm{v}(t,x)|^2\, \mathfrak{b}(x)\big) \big(\mathfrak{b}(x) w_i(x_i) \big) dx_{i}\right) dx_{1},\cdots dx_{i-1}\, dx_{i+1} \cdots dx_{n}.
 \end{split}
\end{equation}
Then, for a fixed $1\leq i \leq n$, we will study the term $\ds{\int_{-1}^{1}  \partial_i (|\mathrm{v}(t,x)|^2\, \mathfrak{b}(x)) \big(\mathfrak{b}(x) w_i(x_i)\big) dx_{i}}$. First, from the explicit definition of the coefficient term $\mathfrak{b}(x)$ given in (\ref{Condition-b}), we write 
\begin{equation}\label{Identity-Tech-1}
\begin{split}
 &\, \int_{-1}^{1}  \partial_i (|\mathrm{v}(t,x)|^2\, \mathfrak{b}(x)) \big(\mathfrak{b}(x) w_i(x_i)\big) dx_{i} \\
=&\, \int_{-1}^{1}  \partial_i \left(|\mathrm{v}(t,x)|^2\, \mathfrak{b}(x)\right)\left( \prod_{j=1}^{n}\mathfrak{b}_j(x_j) w_i(x_i)\right) dx_{i}\\
=&\, \left(\int_{-1}^{1} \partial_i (|\mathrm{v}(t,x)|^2\, \mathfrak{b}(x)) \big(\mathfrak{b}_i(x_i) w_i(x_i)\big) dx_{i} \right) \prod_{j=1,\, j\neq i}^{n} \mathfrak{b}_j(x_j).
\end{split}
\end{equation}
Thereafter, to compute the integral above, we split it into the intervals $[-1,0]$ and $[0,1]$. Note that, from expression (\ref{Component-w}), the the weight $w_i(x_i)$ verifies $w_i(-1)=w_i(1)=0$. Additionally,  from the assumption (\ref{Condition-bi}), the coefficient $\mathfrak{b}_i$ verifies $\mathfrak{b}_i(0)=0$. Therefore,  performing integration by parts and rearranging terms,  we obtain
\begin{equation}\label{Identity-Tech-2}
\begin{split}
&\,  \int_{-1}^{1} \partial_i \big(|\mathrm{v}(t,x)|^2\, \mathfrak{b}(x)\big) \big(\mathfrak{b}_i(x_i)w_i(x_i)\big) dx_{i}
\\
= &\, \int_{-1}^{0}\partial_i \big(|\mathrm{v}(t,x)|^2\, \mathfrak{b}(x)\big)\big(\mathfrak{b}_i(x_i) w_i(x_i)\big) dx_{i} +\int_{0}^{1}\partial_i \big(|\mathrm{v}(t,x)|^2\, \mathfrak{b}(x)\big) \big(\mathfrak{b}_i(x_i)w_i(x_i)\big) dx_{i} 
\\
=&\,  -\int_{-1}^{0}\big(|\mathrm{v}(t,x)|^2\, \mathfrak{b}(x)\big)\,  \mathfrak{b}_i(x_i) \,  w'_i(x_i) dx_{i}-\int_{0}^{1}\big(|\mathrm{v}(t,x)|^2\, \mathfrak{b}(x)\big)\,  \mathfrak{b}_i(x_i) \,  w'_i(x_i) dx_{i}
\\
&\, - \int_{-1}^{1} \big(|\mathrm{v}(t,x)|^2\, \mathfrak{b}(x)\big)\,  \mathfrak{b}'_i(x_i) \,  w_i(x_i) dx_{i}
\\
= &\, 
I_1- \int_{-1}^{1} \big(|\mathrm{v}(t,x)|^2\, \mathfrak{b}(x)\big)\,  \mathfrak{b}'_i(x_i) \,  w_i(x_i) dx_{i}
,
 \end{split}
\end{equation}
where 
\begin{equation*}
    I_1:=  -\int_{-1}^{0}\big(|\mathrm{v}(t,x)|^2\, \mathfrak{b}(x)\big)\,  \mathfrak{b}_i(x_i) \,  w'_i(x_i) dx_{i}-\int_{0}^{1}\big(|\mathrm{v}(t,x)|^2\, \mathfrak{b}(x)\big)\,  \mathfrak{b}_i(x_i) \,  w'_i(x_i) dx_{i}.
\end{equation*}

In the following technical lemma, we derive a lower bound for $I_1$. 
\begin{Lemme} The following estimate holds:
\begin{equation}\label{Estim-Tech-1}
 I_1  \geq \, \frac{\kappa}{2} \int_{-1}^{1} \big(|\mathrm{v}(t,x)|^2\, \mathfrak{b}(x)\big) \mathfrak{b}_i(x_i) w^2_i(x_i)\, dx_i.
\end{equation}
\end{Lemme}
\begin{proof} From  (\ref{Component-w}) it follows that
\begin{equation*}
 w'_i(x_i)=   \begin{cases}\vspace{2mm}
 -\kappa(-x_i)^{-\kappa-1}, \qquad -1<x_i<0, \\
 -\kappa\, x^{-\kappa-1}_{i}, \qquad 0<x_i<1,
    \end{cases}
\end{equation*}
hence, we can write 
\begin{equation}\label{Identity-Tech-1-bis}
 I_1 = \,  \kappa\,\int_{-1}^{0}\big(|\mathrm{v}(t,x)|^2\, \mathfrak{b}(x)\big) \mathfrak{b}_i(x_i)(-x_i)^{-\kappa-1} dx_{i} +\kappa\,\int_{0}^{1}\big(|\mathrm{v}(t,x)|^2\, \mathfrak{b}(x)\big)\mathfrak{b}_i(x_i)\,x^{-\kappa-1}_{i} dx_{i}.  
\end{equation}

On the other hand, always by (\ref{Component-w}), the following pointwise estimates holds
\begin{equation*}
\begin{cases}\vspace{3mm}
 (-x_i)^{-\kappa-1} \geq \frac{1}{2}\,w^2_i(x_i), \qquad -1<x_i<0, \\
 x_i^{-\kappa-1} \geq \frac{1}{2}\,w^2_i(x_i), \qquad 0<x_i<1.
\end{cases}
\end{equation*}
In fact, concerning the first estimate above, for any $-1<x_i<0$, we write $w^2_i(x_i)=((-x_i)^{-\kappa}-1)^2\leq (-x_i)^{-2\kappa}+1$. Since $0<\kappa<1$  and $0<-x_i<1$, it follows that $(-x_i)^{-2\kappa}\leq (-x_i)^{-\kappa-1}$ and $1\leq (-x_i)^{-\kappa-1}$. Thus, we obtain $w^2_i(x_i)\leq 2 (-x_i)^{-\kappa-1}$. The second estimate follows from analogous computations. 

\medskip

Now, taking into account these pointwise estimates and the assumption (\ref{Condition-bi}) that ensures $\mathfrak{b}_i(x_i)\geq 0$, we revisit the identity (\ref{Identity-Tech-1-bis}) to write 
\begin{equation*}
    I_1\geq \frac{\kappa}{2} \int_{-1}^{1} \big(|\mathrm{v}(t,x)|^2\, \mathfrak{b}(x)\big) \mathfrak{b}_i(x_i) w^2_i(x_i)\, dx_i,
\end{equation*}
this last inequality yields the desired estimate (\ref{Estim-Tech-1}). This finishes the proof.
\end{proof}

Thus, considering the estimate (\ref{Estim-Tech-1}) and (\ref{Identity-Tech-2}), we obtain
\begin{equation*}
    \begin{split}
      \int_{-1}^{1} \partial_i \big(|\mathrm{v}(t,x)|^2\, \mathfrak{b}(x)\big) \big(\mathfrak{b}_i(x_i)w_i(x_i)\big) dx_{i} \geq &\, \frac{\kappa}{2} \int_{-1}^{1} \big(|\mathrm{v}(t,x)|^2\, \mathfrak{b}(x)\big) \mathfrak{b}_i(x_i) w^2_i(x_i)\, dx_i\\
      &\, - \int_{-1}^{1} \big(|\mathrm{v}(t,x)|^2\, \mathfrak{b}(x)\big)\,  \mathfrak{b}'_i(x_i) \,  w_i(x_i) dx_{i}.
    \end{split}
\end{equation*}
Then, we come  back to the identity (\ref{Identity-Tech-1}). Recalling  that from the assumption (\ref{Condition-bi}) it follows that $\ds{\prod_{j=1,\, j\neq i}^{n} \mathfrak{b}_j(x_j)\geq 0}$, and from the assumption (\ref{Condition-b}) we have $\ds{\mathfrak{b}(x)=\prod_{j=1}^{n}\mathfrak{b}_j(x_j)}$, we can write

\begin{equation*}
\begin{split}
 &\int_{-1}^{1}  \partial_i (|\mathrm{v}(t,x)|^2\, \mathfrak{b}(x)) \big(\mathfrak{b}(x) w_i(x_i)\big) dx_{i} \\
 \geq &\, \left( \frac{\kappa}{2} \int_{-1}^{1} \big(|\mathrm{v}(t,x)|^2\, \mathfrak{b}(x)\big) \mathfrak{b}_i(x_i) w^2_i(x_i)\, dx_i \right) \prod_{j=1,\, j\neq i}^{n} \mathfrak{b}_j(x_j) 
 \\ &
 \,- \left( \int_{-1}^{1} \big(|\mathrm{v}(t,x)|^2\, \mathfrak{b}(x)\big)\,  \mathfrak{b}'_i(x_i) \,  w_i(x_i) dx_{i} \right)\prod_{j=1,\, j\neq i}^{n} \mathfrak{b}_j(x_j)\\
 =&\,  \frac{\kappa}{2} \int_{-1}^{1} |\mathrm{v}(t,x)|^2\,  \mathfrak{b}^2(x) w^2_i(x_i)\, dx_i
 - \int_{-1}^{1} \big(|\mathrm{v}(t,x)|^2\, \mathfrak{b}(x)\big)\,  \mathfrak{b}'_i(x_i)\, \prod_{j=1,\, j\neq i}^{n} \mathfrak{b}_j(x_j) \,  w_i(x_i) dx_{i}.
 \end{split}
\end{equation*}

With this lower bound at hand,  returning to the identity (\ref{Identity-Tech-0}), we obtain 
\begin{equation}\label{Estim-Tech-3}
\begin{split}
  &\,  \int_{[-1,1]^n} \nabla \big( |\mathrm{v}(t,x)|^2 \mathfrak{b}(x) \big)\cdot \big( \mathfrak{b}(x)\mathrm{w}(x) \big)\, dx \\
 \geq &\, \frac{\kappa}{2} \sum_{i=1}^{n} \int_{-1}^{1}\cdots \left( \int_{-1}^{1} |\mathrm{v}(t,x)|^2 \mathfrak{b}^2(x) w^2_i(x_i)\, dx_i\right)dx_{1}\cdots dx_{i-1}d_{i+1}\cdots dx_n \\
 &\, - \sum_{i=1}^{n} \int_{-1}^{1}\cdots \left(\int_{-1}^{1}|\mathrm{v}(t,x)|^2 \mathfrak{b}(x) \, \mathfrak{b}'_i(x_i)\prod_{j=1,\, j\neq i}^{n} \mathfrak{b}_j(x_j) w_i(x_i)\, dx_i \right)dx_{1}\cdots dx_{i-1}d_{i+1}\cdots dx_n\\
 =&\, \frac{\kappa}{2}\int_{[-1,1]^n} |\mathrm{v}(t,x)|^2 |\mathfrak{b}(x)\mathrm{w}(x)|^2\, dx \\
 &\, - \int_{[-1,1]^n} |\mathrm{v}(t,x)|^2 \mathfrak{b}(x) \left( \sum_{i=1}^n \mathfrak{b}'_i(x_i) w_i(x_i)\prod_{j=1,\, j\neq i}^{n} \mathfrak{b}_j(x_j)\right)dx.
\end{split}
\end{equation}

Here, we define
\[ I_2:= \frac{\kappa}{2}\int_{[-1,1]^n} |\mathrm{v}(t,x)|^2 |\mathfrak{b}(x)\mathrm{w}(x)|^2\, dx,\]
\[I_3:=- \int_{[-1,1]^n} |\mathrm{v}(t,x)|^2 \mathfrak{b}(x) \left( \sum_{i=1}^n \mathfrak{b}'_i(x_i) w_i(x_i)\prod_{j=1,\, j\neq i}^{n} \mathfrak{b}_j(x_j)\right)dx.\]
In the following technical lemmas, we  derive a lower bound for these expressions.
\begin{Lemme} The following lower bound holds:
\begin{equation}\label{Estim-Below-1}
I_2 \geq \frac{\kappa}{2^{n+1}}\left( \int_{[-1,1]^n} \mathrm{v}(t,x)\cdot \big(\mathfrak{b}(x)\mathrm{w}(x)\big)\, dx \right)^2.    
\end{equation}
\end{Lemme}
\begin{proof} Since the Lebesgue measure of the set $[-1,1]^n$ is directly computed as $dx\left( [-1,1]^n \right)=2^n$, by the Jensen inequality it follows that 
\begin{equation*}
    \begin{split}
 I_2 \geq    &\, \frac{\kappa}{2}   \int_{[-1,1]^n} | \mathrm{v}(t,x) \cdot \big( \mathfrak{b}(x)\mathrm{w}(x) \big)|^2\, dx \\ 
\geq &\,\frac{\kappa\, 2^n}{2} \int_{[-1,1]^n} | \mathrm{v}(t,x) \cdot \big( \mathfrak{b}(x)\mathrm{w}(x) \big)|^2\, \frac{dx}{2^n} \\
\geq &\,\frac{\kappa\, 2^n}{2} \left( \int_{[-1,1]^n} \mathrm{v}(t,x)\cdot \big( \mathfrak{b}x \mathrm{w}(x)\big)\, \frac{dx}{2^n} \right)^{2}\\
=&\, \frac{\kappa}{2^{n+1}}\, \left( \int_{[-1,1]^n} \mathrm{v}(t,x)\cdot  \big( \mathfrak{b}(x)\mathrm{w}(x)\big)\,dx \right)^{2}.
    \end{split}
\end{equation*}
\end{proof}
\begin{Lemme} There exists a
constant  ${\bf C}_2(T)={\bf C}_2(n,\kappa, \| \mathfrak{b}\|_{L^\infty}, \| \nabla \mathfrak{b}\|_{L^\infty}, T)>0$ such that following lower bound holds 
\begin{equation}\label{Estim-Below-2}
 I_3 \geq -{\bf C}_2(T)\| u_0\|^2_{H^s}. 
\end{equation}
\end{Lemme}
\begin{proof} Recall that, from  (\ref{Vector-W}),  a (\ref{Component-w}) and  $0<\kappa<1$, it follows that  $\| \mathrm{w}\|_{L^1}\leq C_{n,\kappa}<+\infty$. On the other hand, note that from  the assumption (\ref{Condition-b}) the following inequality holds  $\ds{\| \mathfrak{b}\|_{L^\infty}\leq \sum_{i=1}^n \| \mathfrak{b}_i \|_{L^\infty}}$. Then, using H\"older inequalities ,  we write
\begin{equation*}
\begin{split}
-I_3 = &\,\int_{[-1,1]^n} |\mathrm{v}(t,x)|^2 \mathfrak{b}(x) \left( \sum_{i=1}^n \mathfrak{b}'_i(x_i) w_i(x_i)\prod_{j=1,\, j\neq i}^{n} \mathfrak{b}_j(x_j)\right)dx
\\
\leq &\, \| \mathrm{v}(t,\cdot)\|^2_{L^\infty}\, \| \mathfrak{b}\|_{L^\infty} \left( \int_{[-1,1]^n} | \nabla\mathfrak{b}(x)\cdot \mathrm{w}(x)|\, dx \right) \, \left( \sum_{i=1}^n \prod_{j=1\, j\neq i}^{n}\| \mathfrak{b}_j \|_{L^\infty} \right)\\
\leq &\, \| \mathrm{v}(t,\cdot)\|^2_{L^\infty} \left( \sum_{i=1}^n \| \mathfrak{b}_i \|_{L^\infty}\right) \|\nabla \mathfrak{b}\|_{L^\infty} \, \| \mathrm{w}\|_{L^1}\, n\, \left( \sum_{i=1}^n \| \mathfrak{b}_i \|_{L^\infty}\right)^{n-1}\\
\leq &\, n\, C_{n,\kappa} \left( \sum_{i=1}^n \| \mathfrak{b}_i \|_{L^\infty}\right)^{n}\, \| \nabla\mathfrak{b}\|_{L^\infty}\, \| \mathrm{v}(t,\cdot)\|^2_{L^\infty}.
\end{split}
\end{equation*} 
Additionally, from Sobolev embeddings (with the parameter $n/2<\sigma<s-1$), along with the identity $\mathrm{v}(t,x):=\nabla u(t,x)$ and the control (\ref{Control-Solution}) on $u(t,x)$, it follows that 
\[ \| \mathrm{v}(t,\cdot)\|^2_{L^\infty} \leq C\, \|\mathrm{v}(t,\cdot)\|^2_{H^\sigma}\leq C\, \|u(t,\cdot)\|^2_{H^s}\leq C\, C^2(T)\|u_0\|^2_{H^s}.  \]
Consequently, we obtain
\[-I_3 \leq n\, C_{n,\kappa} \left( \sum_{i=1}^n \| \mathfrak{b}_i \|_{L^\infty}\right)^{n}\, \| \nabla\mathfrak{b}\|_{L^\infty}\,C\, C^2(T)\|u_0\|^2_{H^s}:={\bf C}_{2}(T)\| u_0\|^2_{H^s}.\]
\end{proof}

Finally, the wished lower bound (\ref{Main-Estim-2-n}) directly follows now from  the estimates (\ref{Estim-Below-1}) and (\ref{Estim-Below-2}), along with the inequality (\ref{Estim-Tech-3}). This concludes the proof of Proposition \ref{prop 2 section main thm}. 
\end{proof}

\subsection*{End of the proof}

With the information obtained in estimates 
(\ref{Main-Estim-1-n}) and (\ref{Main-Estim-2-n}), we return to the identity (\ref{Main-Identity-n}) where, for any time $0<T<+\infty$ and for any $0\leq t \leq T$, we  obtain
\begin{equation}\label{ineq. 1 end of proof}
    \begin{split}
       &\,  \frac{d}{dt} \int_{[-1,1]^n} \mathrm{v}(t,x)\cdot \big( \mathfrak{b}(x)\cdot \mathrm{w}(x)\big)\,dx \\
       \geq &\,  - {\bf C}_1(T)\|u_0\|_{H^s}+ \frac{\kappa}{2^{n+1}} \, \left( \int_{[-1,1]^n} \mathrm{v}(t,x)\cdot \big( \mathfrak{b}(x) \mathrm{w}(x) \big)\,dx \right)^{2} -{\bf C}_2 (T)\| u_0\|^2_{H^s}.
  \end{split}
\end{equation}
With this inequality,  following similar ideas of \cite{Jarrin-Vergara-Hermosilla}, we conclude the blow-up in finite of the solution $u(t,x)$ as follows. First, we define the functional
\begin{equation*}
    I(t):=\int_{[-1,1]^n} \mathrm{v}(t,x)\cdot \big( \mathfrak{b}(x) \mathrm{w}(x) \big)\,dx,
\end{equation*}
the constant
\begin{equation*}
    c_1:= \frac{\kappa}{2^{n+1}}>0,
\end{equation*}
and the quantity depending on the time $T$:
\begin{equation*}
   \quad  {\bf C}(T):= {\bf C}_1(T)+{\bf C}_2(T).
\end{equation*}
With this notation, we recast \eqref{ineq. 1 end of proof}
as the following ordinary differential inequality:
\begin{equation*}
   I'(t) \geq c_1 I^2(t)-{\bf C}(T)\times \begin{cases}\vspace{2mm} \| u_0 \|_{H^s}, \quad \text{when}\,\,\, \| u_0\|_{H^s}\leq 1, \\
   \| u_0 \|^2_{H^s}, \quad \text{when}\,\,\, \|u_0\|_{H^s}>1,
   \end{cases}
   \quad 0< t \leq T.
\end{equation*}

Given a number  $m_{*}\gg 1$, we fix the time $T=m_{*}$. So, in the quantity  ${\bf C}(T)$ introduced above  we obtain the numerical constant  ${C}(m_{*})$. Additionally, we define the constant 
\begin{equation*}
    c_2:={\bf C}(m_{*}) \times \begin{cases}\vspace{2mm} \| u_0\|_{H^s}, \quad \text{when}\,\, \|u_0\|_{H^s}\leq 1, \\
       \| u_0\|^2_{H^s}, \quad \text{when}\,\, \|u_0\|_{H^s}>1.
       \end{cases}.
\end{equation*}
Then, the functional $I(t)$ verifies:
\[ I'(t)\geq c_1 I^2(t)-c_2,  \qquad   c_1,c_2>0,  \quad 0<t\leq m_{*},\]
with initial datum given by 
\begin{equation*}
 I(0)=\int_{[-1,1]^n} \mathrm{v}(0,x)\cdot \big( \mathfrak{b}(x) \mathrm{w}(x)\big)\,dx=\int_{[-1,1]^n} \nabla u_0(x)\cdot \big(\mathfrak{b}(x)\mathrm{w}(x)\big)\,dx.
\end{equation*}

\medskip

To continue, we denote by $J(t)$ the solution of the initial value problem below
\begin{equation*}
\begin{cases}\vspace{2mm}
   J'(t)= c_1 J^2(t)-c_2, \\
   J(0)=I(0),
\end{cases}
 \end{equation*}
which, is explicitly computed as:
\begin{equation*}
 J(t)= \sqrt{\frac{c_2}{c_1}}\, \frac{\left( I(0)+ \sqrt{\frac{c_2}{c_1}} \right)+\left( I(0) -  \sqrt{\frac{c_2}{c_1}} \right) e^{2 \sqrt{c_1c_2}\, t}}{\left( I(0)+ \sqrt{\frac{c_2}{c_1}} \right)-\left( I(0) -  \sqrt{\frac{c_2}{c_1}} \right) e^{2 \sqrt{c_1c_2}\, t}}.
\end{equation*}

Note that this function blows-up at the finite time
\begin{equation}\label{Time-blow-up}
    t_{*}:=\frac{1}{2\sqrt{c_1c_2}}\, \ln\left(  \frac{I(0)\sqrt{c_1}+\sqrt{c_2}}{I(0)\sqrt{c_1}-\sqrt{c_2}} \right),
\end{equation} 
which is positive  as long as the following inequality holds:
\[ \frac{I(0)\sqrt{c_1}+\sqrt{c_2}}{I(0)\sqrt{c_1}-\sqrt{c_2}} >1. \]
From  the conditions \eqref{Condition-u0-1} and \eqref{Condition-u0-2}  assumed on the initial datum $u_0$, together with the expressions $I(0)$, $c_1$ and $c_2$ defined above, we obtain that $I(0)>0$ and  $\ds{I(0)\sqrt{c_1}-\sqrt{c_2}>0}$, respectively. Thus,  the wished inequality holds. 

\medskip

Now, by the standard comparison principle of ordinary differential equations, it follows that $I(t)\geq J(t)$, and the functional $I(t)$ also blows-up at the time $t_{*}$. 

\medskip

Finally, recalling the bound  $\| \mathrm{w}\|_{L^1([-1,1]^n)}\leq C_{n,\kappa}<+\infty$,  and considering the H\"older  inequality, we can write
\begin{equation*}
\begin{split}
     I(t)&  \leq \| \mathrm{v}(t,\cdot)\|_{L^\infty([-1,1]^n)}\, \| \mathfrak{b}\|_{L^\infty([-1,1]^n)} \, \|  \mathrm{w}\|_{L^1([-1,1]^n)} 
     \\
      & \leq C_{n,\kappa}\, \| \mathfrak{b}\|_{L^\infty}\, \| \mathrm{v}(t,\cdot)\|_{L^\infty(\R^n)} 
      . 
    \end{split}
\end{equation*}
Thus, by the   
Sobolev embeddings (with a parameter $n/2<\sigma<s-1$), we obtain  
\begin{equation*}
\begin{split}
     I(t)&  \leq  C_{n,\kappa} \, \| \mathfrak{b}\|_{L^\infty}\,\|\mathrm{v}(t,\cdot) \|_{H^\sigma}
      \\
      &\leq 
      C_{n,\kappa} \,\| \mathfrak{b}\|_{L^\infty}\,\| u(t,\cdot)\|_{H^{\sigma+1}}\\
      &\leq  C_{n,\kappa} \,\| \mathfrak{b}\|_{L^\infty}\,\| u(t,\cdot)\|_{H^s}. 
    \end{split}
\end{equation*}

Consequently, the quantity $\| u(t,\cdot)\|_{H^s}$ blows-up at the finite time $t_{*}$, which gives a contradiction. This concludes the proof of
Theorem \ref{Main-Th}.

\medskip 

\paragraph{\bf Acknowledgements.} 
The authors warmly thank Professor Rafael Granero-Belinchón for suggesting this method to study blow-up phenomenon. We also thank  Professors Alexey Cheskidov and Philippe Laurençot for their helpful comments and advises.

\medskip

\paragraph{{\bf Statements and Declaration}}
Data sharing does not apply to this article as no datasets were generated or analyzed during the current study.  In addition, the authors declare that they have no conflicts of interest, and all of them have equally contributed to this paper.

\end{document}